\documentclass[oneside,reqno,english]{amsart}
\usepackage[T1]{fontenc}
\usepackage[utf8]{inputenc}
\usepackage{xcolor}
\usepackage{babel}
\usepackage{prettyref}
\usepackage{amstext}
\usepackage{amsthm}
\usepackage{amssymb}
\usepackage[pdfusetitle,
 bookmarks=true,bookmarksnumbered=false,bookmarksopen=false,
 breaklinks=false,pdfborder={0 0 0},pdfborderstyle={},backref=false,colorlinks=false]
 {hyperref}
\hypersetup{
 colorlinks=true,citecolor=blue,linkcolor=blue,linktocpage=true}

\makeatletter
\numberwithin{equation}{section}
\numberwithin{figure}{section}

\usepackage{prettyref}

\newrefformat{cor}{Corollary~\ref{#1}}
\newrefformat{subsec}{Section~\ref{#1}}
\newrefformat{lem}{Lemma~\ref{#1}}
\newrefformat{thm}{Theorem~\ref{#1}}
\newrefformat{sec}{Section~\ref{#1}}
\newrefformat{chap}{Chapter~\ref{#1}}
\newrefformat{prop}{Proposition~\ref{#1}}
\newrefformat{exa}{Example~\ref{#1}}
\newrefformat{tab}{Table~\ref{#1}}
\newrefformat{rem}{Remark~\ref{#1}}
\newrefformat{def}{Definition~\ref{#1}}
\newrefformat{fig}{Figure~\ref{#1}}
\newrefformat{claim}{Claim~\ref{#1}}
\newrefformat{assu}{Assumption~\ref{#1}}

\makeatother

\theoremstyle{plain}
\newtheorem{thm}{\protect\theoremname}[section]
\newtheorem{lem}[thm]{\protect\lemmaname}
\newtheorem{prop}[thm]{\protect\propositionname}
\theoremstyle{remark}
\newtheorem{rem}[thm]{\protect\remarkname}
\theoremstyle{definition}
\newtheorem{defn}[thm]{\protect\definitionname}
\theoremstyle{plain}
\newtheorem{cor}[thm]{\protect\corollaryname}
\providecommand{\corollaryname}{Corollary}
\providecommand{\definitionname}{Definition}
\providecommand{\lemmaname}{Lemma}
\providecommand{\propositionname}{Proposition}
\providecommand{\remarkname}{Remark}
\providecommand{\theoremname}{Theorem}

\begin{document}
\title[Tree Coordinates and Range Martingales for POVMs]{Tree Coordinates and Range Martingales for Positive Operator-Valued
Measures}
\begin{abstract}
Positive operator-valued measures on a tree admit intrinsic local
coordinates coming from the way each cylinder value splits into its
children. We show that these local splittings, taken on the range
spaces of the cylinder values, recover the measure and at the same
time build an intrinsic direct limit dilation whose cylinder projections
yield the minimal Naimark dilation. In these coordinates, the commutant
of the dilation becomes a martingale calculus on the range spaces.
This gives local descriptions of extremality and domination, and it
also yields a bounded change-of-measure transform that updates the
tree coordinates in a natural way. For self-adjoint range martingales
we obtain a quadratic variation formula from the range space isometries,
and the associated local variance terms detect the projection-valued
case.
\end{abstract}

\author{James Tian}
\address{Mathematical Reviews, 535 W. William St, Suite 210, Ann Arbor, MI
48103, USA}
\email{james.ftian@gmail.com}

\maketitle
\tableofcontents{}

\section{Introduction}\label{sec:1}

Positive operator-valued measures (POVMs) are a basic object in operator
theory and in quantum information, where they model generalized measurements.
When the underlying sample space has a tree structure, one expects
a local description in terms of how the measure splits from a parent
cylinder to its children. For scalar measures this is elementary.
Each parent mass is distributed among the children by conditional
probabilities. For operator-valued measures there is no comparable
scalar coordinate system on the original Hilbert space, because the
relevant subspace changes from one cylinder to the next. The starting
point of this paper is that the right local coordinates live on the
range spaces of the cylinder values.

If one passes from a cylinder value to the closure of the range of
its square root, then each local splitting is encoded by a positive
contraction on the current range space. These splitting operators
recover the measure recursively along the tree, and they reduce to
the familiar conditional weights in the scalar and commuting cases
(\prettyref{prop:2-2}, \prettyref{rem:2-3}). This gives a coordinate
system for tree indexed POVMs. The projection-valued case also becomes
local in these coordinates, since this is equivalent to asking that
the splitting operators be projections on the current range spaces
(\prettyref{prop:2-4}).

This range space description leads to a canonical dilation. From the
local splittings one obtains edge contractions between the range spaces
at adjacent levels, and these contractions assemble into a direct
limit Hilbert space. The resulting cylinder projections recover the
original measure by compression, and this direct limit construction
gives the minimal Naimark dilation (Theorems \ref{thm:3-3}, \ref{thm:3-4}).
In the usual presentation, a dilation is often taken as external data
from which one studies the measure. Here the direction is reversed.
The measure itself carries enough local structure to build its minimal
dilation.

Once the dilation is written in these coordinates, its commutant admits
a tree martingale description. The relevant objects are bounded self-adjoint
operators on the range spaces that satisfy a local averaging relation
through the edge contractions. We call these range martingales. They
turn the commutant of the dilation into an  object attached to the
measure (\prettyref{thm:4-4}). This has two immediate consequences.
Extremality becomes a statement that there are no nonzero range martingales
vanishing at the root (\prettyref{thm:4-5}), and domination by the
given measure is described by positive range martingales with prescribed
bounds and root value (\prettyref{thm:5-2}, \prettyref{cor:5-3}).
In this way, two standard themes in the theory of POVMs, convex structure
and Radon-Nikodym type domination, acquire a local description on
the tree.

The same coordinates also support a bounded change-of-measure calculus.
A strictly positive normalized range martingale produces a new measure
that is boundedly equivalent to the original one, and the local splitting
operators transform by a conjugation formula along the tree (\prettyref{thm:6-3},
\prettyref{prop:6-5}). In the scalar case this is the usual Doob
transform. In the operator-valued setting it gives a way to compare
nearby measurements while staying within the same  coordinate system.
This is natural both in operator-valued measure theory and in quantum
information, where one often studies how a measurement changes under
reweighting, post-processing, or successive conditioning along finite
outcome strings.

The final part of the paper shows that the martingale picture also
carries a square-function calculus. Each bounded self-adjoint range
martingale has a local variance term at every vertex, and these local
terms add up to an increasing quadratic variation along the tree (\prettyref{lem:7-2},
\prettyref{thm:7-3}). When this construction is applied to the child
coordinate projections, the resulting local variance is the difference
between a splitting operator and its square. The corresponding sum
measures how far the local splitting is from the projection-valued
case, and vanishes if and only if the measure is projection-valued
(\prettyref{prop:7-5}). Thus the same local coordinates that encode
the measure, build the dilation, and describe domination also provide
a local notion of variance that detects sharpness.

\textbf{Literature context.} This paper sits near several parts of
operator theory and mathematical quantum theory. The dilation-theoretic
background comes from Stinespring's theorem and Arveson's work on
subalgebras of $C^{*}$-algebras, completely positive maps, and noncommutative
Choquet theory \cite{MR253059,MR2425180,MR2641960,MR2823981,MR2743416}.
Related pure operator-theoretic developments include injective envelopes
and operator spaces, boundary representations, and the Choquet boundary
of operator systems \cite{MR566081,MR1869114,MR2132691,MR3430455,MR4403230,MR5009783}.
The present construction uses this dilation background, but the dilation
is not taken as external data. It is built from the range spaces of
the cylinder values.

There is also a substantial literature on operator-valued measures
and POVMs as mathematical objects. This includes work on clean measurements,
randomness and conditional expectation for quantum random variables,
commutative POVMs, Feller Markov kernels, Lyapunov theorems, operator-valued
frames, and dilation theory for operator-valued measures \cite{MR2165832,MR2321233,MR2907638,MR2953266,MR2548409,MR2548410,MR3498318,MR3857513,MR4137283,MR4280112,MR3778680,MR4162409,MR4735867}.
These papers give several ways of studying POVMs through order, dilation,
randomness, Markov kernels, and frame theoretic structure. The tree
setting considered here leads to a different local description, in
which each cylinder value is replaced by its square-root range space
and each split into child cylinders is represented by a positive decomposition
of the identity on that range space.

The convex structure of POVMs has been studied from several directions,
including extreme observables, $C^{*}$-extreme points, covariant
measurements, relabeling and mixing, and extreme marginals of completely
positive maps \cite{MR1810813,MR2105199,MR2262678,MR2432912,MR2770378,MR2817979,MR2930514,MR2990784,MR3133476,MR3190205,MR3339203,MR4112227,MR4340930}.
The extremality result in this paper uses the standard commutant criterion,
but expresses it in the range coordinates as the absence of nonzero
bounded range martingales with zero root value. In the same language,
dominated positive operator-valued measures become positive range
martingales.

The Radon-Nikodym theorem used in the paper is also connected with
earlier work on completely positive maps, quantum operations, and
quantum instruments \cite{MR932932,MR1608473,MR2014842}. What is
added here is the local tree form of that theory. A dominated measure
is encoded by a positive range martingale, and a boundedly equivalent
change of measure updates the edge contractions by an explicit noncommutative
Doob transform. This is the operator-valued analogue of changing transition
probabilities on a filtered tree by a positive scalar martingale.

The paper also touches the broader mathematical literature on quantum
measurements, compatibility, incompatibility, and repeated or nondisturbing
measurements \cite{MR2742801,MR3126873,MR3417980,MR3787334,MR4042565,MR4240433}.
These questions depend on how one compares, refines, or modifies a
measurement. The range space coordinates developed here give a local
operator-theoretic way to carry out such comparisons on a tree: domination
becomes positivity of a range martingale, bounded change of measure
becomes an explicit update of the edge contractions, and the quadratic
variation terms measure the local departure from the projection-valued
case.

\section{Tree coordinates}\label{sec:2}

This section introduces the local coordinates used throughout the
paper. For a tree indexed POVM, each child cylinder value is dominated
by the value at its parent, so one can describe the passage from parent
to child by a positive contraction on the parent range space. These
operators will serve as the local tree coordinates. They recover the
cylinder values recursively, reduce to the usual conditional weights
in the scalar case, and give a local description of the projection-valued
case. 

This construction is related to two standard ways of treating POVMs.
One is to regard a POVM as an operator-valued measure and study its
order, range, and dilation properties \cite{MR4137283,MR4280112,MR3778680}.
Another is to compare POVMs with classical post-processings and Markov
kernels in the commutative case \cite{MR2548409,MR2548410,MR3498318}.
The local coordinates below keep the parent-child conditioning structure
of the scalar case, but place the local data on the range spaces. 

We begin with the binary tree, where the notation is simplest, and
then pass to the finite-alphabet version. Let 
\begin{equation}
\Omega=\left\{ 0,1\right\} ^{\mathbb{N}}\label{eq:2-1}
\end{equation}
with its product Borel $\sigma$-algebra $\mathcal{B}\left(\Omega\right)$.
If $w=w_{1}\cdots w_{n}$ is a finite word in the alphabet $\left\{ 0,1\right\} $,
let 
\begin{equation}
\left[w\right]=\left\{ x\in\Omega:x_{1}=w_{1},\ldots,x_{n}=w_{n}\right\} \label{eq:2-2}
\end{equation}
be the corresponding cylinder set. The empty word is denoted by $\varnothing$,
so that $\left[\varnothing\right]=\Omega$.

Let $H$ be a Hilbert space. A POVM on $\Omega$ is a countably additive
map 
\begin{equation}
E:\mathcal{B}\left(\Omega\right)\to B\left(H\right)_{+}\label{eq:2-3}
\end{equation}
such that $E\left(\Omega\right)=I$. Countable additivity is understood
in the weak operator topology. By definition, $E_{\varnothing}=E\left(\Omega\right)=I$,
and $E\left(\emptyset\right)=0$.

For each finite word $w$, write 
\begin{equation}
E_{w}=E\left(\left[w\right]\right).\label{eq:2-4}
\end{equation}
Then, for every $w$, 
\begin{equation}
E_{w}=E_{w0}+E_{w1}.\label{eq:2-5}
\end{equation}

We shall use the following form of the Radon-Nikodym lemma for positive
operators.
\begin{lem}
\label{lem:2-1} Let $0\leq F\leq G$ in $B\left(H\right)$, and set
\begin{equation}
H_{G}=\overline{ran}\,(G^{1/2}).\label{eq:2-6}
\end{equation}
Then there is a unique positive contraction $A$ on $H_{G}$ such
that 
\begin{equation}
F=G^{1/2}AG^{1/2},\label{eq:2-7}
\end{equation}
where $A$ is extended by zero on $H^{\perp}_{G}$. 
\end{lem}

\begin{proof}
By Douglas factorization \cite{MR203464}, there is a contraction
$C\in B\left(H\right)$ such that $F^{1/2}=CG^{1/2}$. Hence 
\[
F=G^{1/2}C^{*}CG^{1/2}.
\]
Let $P_{G}$ be the orthogonal projection onto $H_{G}$, and define
\[
A=\left.P_{G}C^{*}CP_{G}\right|_{H_{G}}.
\]
Then $0\leq A\leq I_{H_{G}}$, and since $G^{1/2}=P_{G}G^{1/2}$,
we have \prettyref{eq:2-7}.

If $A_{1}$ and $A_{2}$ are positive contractions on $H_{G}$ satisfying
\prettyref{eq:2-7}, then 
\[
\left\langle G^{1/2}x,\left(A_{1}-A_{2}\right)G^{1/2}y\right\rangle =0
\]
for all $x,y\in H$. Since $ran\,(G^{1/2})$ is dense in $H_{G}$,
it follows that $A_{1}=A_{2}$. 
\end{proof}

\begin{prop}
\label{prop:2-2} Let $E$ be a POVM on $\Omega$, and let $E_{w}$
be defined by \prettyref{eq:2-4}. Then, for every finite word $w$,
there is a unique positive contraction $A_{w}$ on 
\[
H_{w}=\overline{ran}\,(E^{1/2}_{w})
\]
such that 
\begin{equation}
E_{w0}=E^{1/2}_{w}A_{w}E^{1/2}_{w},\label{eq:2-8}
\end{equation}
and 
\begin{equation}
E_{w1}=E^{1/2}_{w}\left(I_{H_{w}}-A_{w}\right)E^{1/2}_{w}.\label{eq:2-9}
\end{equation}

Conversely, suppose that $E_{\varnothing}=I$, and that the cylinder
values $E_{w}$ are obtained recursively from positive contractions
$A_{w}$ by \prettyref{eq:2-8} and \prettyref{eq:2-9}. Then there
is a unique POVM, $\mathcal{B}\left(\Omega\right)\xrightarrow{\;E\;}B\left(H\right)_{+}$
with these cylinder values. 
\end{prop}

\begin{proof}
Let $E$ be a POVM. By \prettyref{eq:2-5}, $0\leq E_{w0}\leq E_{w}$.
Applying \prettyref{lem:2-1} with $F=E_{w0}$ and $G=E_{w}$ gives
\prettyref{eq:2-8} and the uniqueness of $A_{w}$. Since $E_{w1}=E_{w}-E_{w0}$,
\prettyref{eq:2-9} follows.

For the converse, the recursive construction gives positive cylinder
values satisfying \prettyref{eq:2-5}. For every $h\in H$, define
\[
\mu_{h}\left(\left[w\right]\right)=\left\langle h,E_{w}h\right\rangle .
\]
Then $\mu_{h}$ gives a finite positive measure on each finite-level
cylinder algebra, and these finite-level measures are compatible under
refinement. Since 
\[
\mu_{h}\left(\Omega\right)=\left\Vert h\right\Vert ^{2},
\]
Kolmogorov extension gives a unique finite positive Borel measure
on $\Omega$, still denoted by $\mu_{h}$.

For $h,k\in H$, define 
\[
\mu_{h,k}=\frac{1}{4}\sum^{3}_{\ell=0}\left(-i\right)^{\ell}\mu_{h+i^{\ell}k}.
\]
On cylinders, 
\[
\mu_{h,k}\left(\left[w\right]\right)=\left\langle h,E_{w}k\right\rangle .
\]
The scalar identities in the polarization formula pass to Borel sets
by uniqueness of scalar measure extension. Thus, for each Borel set
$B\subset\Omega$, the map 
\[
\left(h,k\right)\mapsto\mu_{h,k}\left(B\right)
\]
is sesquilinear. It is bounded because 
\[
\left|\mu_{h,k}\left(B\right)\right|^{2}\leq\mu_{h}\left(B\right)\mu_{k}\left(B\right)\leq\left\Vert h\right\Vert ^{2}\left\Vert k\right\Vert ^{2}.
\]
Hence there is a unique operator $E\left(B\right)\in B\left(H\right)$
such that 
\[
\mu_{h,k}\left(B\right)=\left\langle h,E\left(B\right)k\right\rangle .
\]
These operators form a POVM, and their cylinder values are the prescribed
operators $E_{w}$. 
\end{proof}

The positive contractions $A_{w}$ in \prettyref{prop:2-2} are the
local coordinates of $E$ along the binary tree. We shall call them
the splitting operators.
\begin{rem}
\label{rem:2-3}For scalar measures, \prettyref{prop:2-2} reduces
to the usual conditional splitting along the binary tree. If $\nu$
is a probability measure on $\Omega$, write $\nu_{w}=\nu\left(\left[w\right]\right)$.
Whenever $\nu_{w}>0$, there is a unique $p_{w}\in\left[0,1\right]$
such that 
\[
\nu_{w0}=p_{w}\nu_{w},\qquad\nu_{w1}=\left(1-p_{w}\right)\nu_{w}.
\]
If $\nu_{w}=0$, then both children have zero mass and $p_{w}$ may
be chosen arbitrarily. This scalar recursion can be written as 
\begin{equation}
\nu_{w0}=\nu^{1/2}_{w}p_{w}\nu^{1/2}_{w},\qquad\nu_{w1}=\nu^{1/2}_{w}\left(1-p_{w}\right)\nu^{1/2}_{w}.\label{eq:2-10}
\end{equation}
Thus the scalar number $p_{w}$ is replaced in \prettyref{prop:2-2}
by a positive contraction $A_{w}$ on $H_{w}=\overline{ran}\,(E^{1/2}_{w})$,
and \prettyref{eq:2-8} and \prettyref{eq:2-9} replace the two scalar
formulas in \prettyref{eq:2-10}.

If $H=\mathbb{C}$, this gives the scalar recursion. If the operators
$E_{w}$ commute, the same reduction occurs after simultaneous diagonalization.
Indeed, in a multiplication representation, $E_{w}=M_{f_{w}}$ and
$E_{w0}=M_{f_{w0}}$ with $0\leq f_{w0}\leq f_{w}$, and the splitting
operator is multiplication by 
\[
a_{w}=\begin{cases}
f_{w0}/f_{w}, & f_{w}>0,\\
0, & f_{w}=0.
\end{cases}
\]
Then 
\[
f_{w0}=f_{w}a_{w},\qquad f_{w1}=f_{w}\left(1-a_{w}\right).
\]
\end{rem}

The same coordinates also identify the projection-valued case. For
scalar measures, the corresponding condition is that the splitting
numbers take only the values $0$ and $1$. For operator-valued measures,
the scalar $0$-$1$ condition is replaced by the condition that each
local splitting operator is a projection on the current range space.
\begin{prop}
\label{prop:2-4} Let $E$ be a POVM on $\Omega$, and let $A_{w}$
be its splitting operators. Then $E$ is projection-valued if and
only if $A^{2}_{w}=A_{w}$ for every finite word $w$. 
\end{prop}

\begin{proof}
Suppose first that $E$ is projection-valued. Then each $E_{w}$ is
an orthogonal projection, and $H_{w}=E_{w}H$. Since $\left[w0\right]\subseteq\left[w\right]$,
we have $E_{w0}\leq E_{w}$. Hence $E_{w}E_{w0}=E_{w0}$. Also $E^{1/2}_{w}=E_{w}$.
By the defining identity for $A_{w}$, 
\[
E_{w0}=E_{w}A_{w}E_{w}.
\]
Since $A_{w}$ acts on $H_{w}$ and is extended by zero on $H^{\perp}_{w}$,
this gives 
\[
A_{w}=\left.E_{w0}\right|_{H_{w}}.
\]
Thus $A_{w}$ is an orthogonal projection on $H_{w}$.

Conversely, suppose that every $A_{w}$ is an orthogonal projection.
We prove by induction on $\left|w\right|$ that each $E_{w}$ is an
orthogonal projection. This is true for $w=\varnothing$, since $E_{\varnothing}=I$.
Assume that $E_{w}$ is a projection. Then $H_{w}=E_{w}H$ and $E^{1/2}_{w}=E_{w}$.
Since $A_{w}$ is a projection on $H_{w}$, 
\[
E_{w0}=E_{w}A_{w}E_{w}=A_{w}
\]
as an operator on $H$, with $A_{w}$ extended by zero on $H^{\perp}_{w}$.
Hence $E_{w0}$ is a projection. Similarly, 
\[
E_{w1}=E_{w}\left(I_{H_{w}}-A_{w}\right)E_{w}=E_{w}-A_{w},
\]
so $E_{w1}$ is also a projection. This proves the induction step.

It remains to pass from cylinders to Borel sets. For finite words
$u$ and $v$, the projections $E_{u}$ and $E_{v}$ satisfy 
\[
E_{u}E_{v}=E\left(\left[u\right]\cap\left[v\right]\right).
\]
Indeed, if one of $u,v$ extends the other, this follows from the
order relation between the corresponding cylinder projections. If
the two words are incompatible, the cylinder sets are disjoint and
the corresponding projections are orthogonal.

By finite additivity, the identity 
\[
E\left(B\cap C\right)=E\left(B\right)E\left(C\right)
\]
holds whenever $B$ and $C$ are finite unions of cylinders. A standard
monotone class argument, first in $B$ with $C$ fixed and then in
$C$ with $B$ fixed, extends this identity to all Borel sets $B,C\subseteq\Omega$.
Taking $C=B$ gives 
\[
E\left(B\right)^{2}=E\left(B\right).
\]
Since $E\left(B\right)$ is positive, it is an orthogonal projection.
Thus $E$ is projection-valued.
\end{proof}

\begin{rem}
\label{rem:2-5} In the scalar case, the projection-valued condition
means that each local splitting number is either $0$ or $1$. Thus
each nonzero cylinder sends all of its mass to one child. The operator-valued
case is different. The condition in \prettyref{prop:2-4} only requires
$A_{w}$ to be a projection on $H_{w}$; it does not require $A_{w}$
to be either $0$ or $I_{H_{w}}$. A projection-valued measure (PVM)
may split 
\[
H_{w}=A_{w}H_{w}\oplus\left(I_{H_{w}}-A_{w}\right)H_{w}
\]
with both summands nonzero. Thus the scalar $0$-$1$ condition is
replaced with an orthogonal decomposition of the current range space. 
\end{rem}

The binary tree case has a finite-alphabet version. The symmetric
form is not obtained by ordering the children and making repeated
binary choices. Instead, each vertex carries a finite positive decomposition
of the current range space.
\begin{thm}
\label{thm:2-6} Let $m\geq2$, and let 
\begin{equation}
\Omega_{m}=\left\{ 0,\ldots,m-1\right\} ^{\mathbb{N}}\label{eq:2-11}
\end{equation}
with its product Borel structure. If $E$ is a POVM on $\Omega_{m}$,
write 
\begin{equation}
E_{w}=E\left(\left[w\right]\right),\qquad H_{w}=\overline{ran}\,(E^{1/2}_{w})\label{eq:2-12a}
\end{equation}
for each finite word $w$ in the alphabet $\left\{ 0,\ldots,m-1\right\} $.
Then, for every $w$, there are unique positive contractions 
\begin{equation}
A^{\left(0\right)}_{w},\ldots,A^{\left(m-1\right)}_{w}\in B\left(H_{w}\right)\label{eq:2-12}
\end{equation}
such that 
\begin{equation}
\sum^{m-1}_{j=0}A^{\left(j\right)}_{w}=I_{H_{w}}\label{eq:2-13}
\end{equation}
and 
\begin{equation}
E_{wj}=E^{1/2}_{w}A^{\left(j\right)}_{w}E^{1/2}_{w}.\label{eq:2-14}
\end{equation}

Conversely, suppose that $E_{\varnothing}=I$, and that the cylinder
values are defined recursively by positive contractions $A^{\left(j\right)}_{w}\in B\left(H_{w}\right)$
satisfying \prettyref{eq:2-13} and \prettyref{eq:2-14}. Then there
is a unique POVM $E$ on $\Omega_{m}$ with these cylinder values. 
\end{thm}

\begin{proof}
Let $E$ be a POVM on $\Omega_{m}$. For every finite word $w$, countable
additivity gives 
\begin{equation}
E_{w}=\sum^{m-1}_{j=0}E_{wj}.\label{eq:2-15}
\end{equation}
In particular, $0\leq E_{wj}\leq E_{w}$ for each $j$. Applying \prettyref{lem:2-1}
to $F=E_{wj}$ and $G=E_{w}$ gives a unique positive contraction
$A^{\left(j\right)}_{w}$ satisfying \prettyref{eq:2-14}.

It remains to prove \prettyref{eq:2-13}. Summing \prettyref{eq:2-14}
over $j$ and using \prettyref{eq:2-15} gives 
\[
E^{1/2}_{w}\left(\sum\nolimits^{m-1}_{j=0}A^{\left(j\right)}_{w}-I_{H_{w}}\right)E^{1/2}_{w}=0.
\]
Thus 
\[
\left\langle E^{1/2}_{w}x,\left(\sum\nolimits^{m-1}_{j=0}A^{\left(j\right)}_{w}-I_{H_{w}}\right)E^{1/2}_{w}y\right\rangle =0
\]
for all $x,y\in H$. Since $ran\,(E^{1/2}_{w})$ is dense in $H_{w}$,
\prettyref{eq:2-13} follows.

Conversely, assume that the cylinder values are recursively defined
by \prettyref{eq:2-13} and \prettyref{eq:2-14}. Then each $E_{wj}$
is positive, and summing \prettyref{eq:2-14} over $j$ gives \prettyref{eq:2-15}.
Hence the cylinder values are compatible under refinement. The scalarization
and polarization argument in \prettyref{prop:2-2} then gives a unique
POVM on $\Omega_{m}$ with these cylinder values. 
\end{proof}

\begin{defn}
\label{def:2-7} For a POVM $E$ on $\Omega_{m}$, the local tree
coordinates of $E$ are the positive contractions in \prettyref{eq:2-12}
given by \prettyref{thm:2-6}. We shall also call them the splitting
operators of $E$. 
\end{defn}

\section{Range space dilation}\label{sec:3}

The local splitting operators from \prettyref{sec:2} do more than
recover the cylinder values. In this section, we show that they give
a direct construction of a Naimark dilation. This is in the spirit
of dilation theory for positive and completely positive maps, especially
Arveson's dilation and extension framework \cite{MR253059,MR2425180,MR2641960,MR2823981},
and of later work on dilations of operator-valued measures \cite{MR3778680}.

Throughout this section let $m\geq2$, and let $\Omega_{m}$ be as
in \prettyref{eq:2-11}. Let $E$ be a POVM on $\Omega_{m}$, and
let 
\[
E_{w}=E\left(\left[w\right]\right),\qquad H_{w}=\overline{ran}\,(E^{1/2}_{w}).
\]
For each finite word $w$ and each $0\leq j\leq m-1$, define 
\[
C_{wj}:H_{w}\to H_{wj}
\]
on the dense subspace $ran\,(E^{1/2}_{w})\subseteq H_{w}$ by 
\begin{equation}
C_{wj}E^{1/2}_{w}h=E^{1/2}_{wj}h.\label{eq:3-1}
\end{equation}
This is well-defined and contractive. Indeed, if $E^{1/2}_{w}h=0$,
then $E^{1/2}_{wj}h=0$, since $0\leq E_{wj}\leq E_{w}$. Moreover,
\[
\Vert E^{1/2}_{wj}h\Vert^{2}=\left\langle h,E_{wj}h\right\rangle \leq\left\langle h,E_{w}h\right\rangle =\Vert E^{1/2}_{w}h\Vert^{2}.
\]
Thus \prettyref{eq:3-1} extends to a contraction from $H_{w}$ to
$H_{wj}$.
\begin{lem}
\label{lem:3-1} For every finite word $w$ and every $0\leq j\leq m-1$,
\[
C^{*}_{wj}C_{wj}=A^{\left(j\right)}_{w},
\]
where $A^{\left(j\right)}_{w}$ is the local splitting operator from
\prettyref{thm:2-6}. 
\end{lem}

\begin{proof}
For $h,k\in H$, \prettyref{eq:3-1} gives 
\[
\left\langle C_{wj}E^{1/2}_{w}h,C_{wj}E^{1/2}_{w}k\right\rangle =\left\langle E^{1/2}_{wj}h,E^{1/2}_{wj}k\right\rangle =\left\langle h,E_{wj}k\right\rangle .
\]
By \prettyref{eq:2-14}, 
\[
\left\langle h,E_{wj}k\right\rangle =\left\langle E^{1/2}_{w}h,A^{\left(j\right)}_{w}E^{1/2}_{w}k\right\rangle .
\]
Since $ran\,(E^{1/2}_{w})$ is dense in $H_{w}$, the result follows. 
\end{proof}

For each finite word $w$, define 
\begin{equation}
J_{w}:H_{w}\longrightarrow\bigoplus^{m-1}_{j=0}H_{wj},\qquad J_{w}x=\left(C_{w0}x,\ldots,C_{w,m-1}x\right).\label{eq:3-2}
\end{equation}
By \prettyref{lem:3-1} and \prettyref{eq:2-13}, 
\[
J^{*}_{w}J_{w}=\sum^{m-1}_{j=0}C^{*}_{wj}C_{wj}=\sum^{m-1}_{j=0}A^{\left(j\right)}_{w}=I_{H_{w}}.
\]
Hence $J_{w}$ is an isometry.

For $n\geq0$, set 
\begin{equation}
K_{n}=\bigoplus_{\left|w\right|=n}H_{w}.\label{eq:3-3}
\end{equation}
The maps $J_{w}$ assemble into an isometry 
\[
J_{n}:K_{n}\to K_{n+1}
\]
defined by 
\begin{equation}
\left(J_{n}x\right)_{wj}=C_{wj}x_{w},\qquad\left|w\right|=n,\quad0\leq j\leq m-1.\label{eq:3-4}
\end{equation}
Let $K_{E}$ be the Hilbert space direct limit of the inductive system
\begin{equation}
K_{0}\xrightarrow{J_{0}}K_{1}\xrightarrow{J_{1}}K_{2}\xrightarrow{J_{2}}\cdots.\label{eq:3-5}
\end{equation}
Let 
\[
U_{n}:K_{n}\to K_{E}
\]
be the canonical isometries. Thus 
\[
U_{n+1}J_{n}=U_{n},
\]
and 
\[
K_{E}=\overline{\bigcup_{n\geq0}U_{n}K_{n}}.
\]
Since $E_{\varnothing}=I$, we have $K_{0}=H$. We write 
\begin{equation}
V=U_{0}:H\to K_{E}.\label{eq:3-6}
\end{equation}

For $0\leq r\leq n$, let 
\[
J_{r,n}=J_{n-1}\cdots J_{r}:K_{r}\to K_{n},
\]
with $J_{n,n}=I_{K_{n}}$.
\begin{lem}
\label{lem:3-2} For every $n\geq0$ and every $h\in H$, 
\[
\left(J_{0,n}h\right)_{w}=E^{1/2}_{w}h,\qquad\left|w\right|=n.
\]
\end{lem}

\begin{proof}
The assertion is clear for $n=0$. If it holds at level $n$, then
for $\left|w\right|=n$ and $0\leq j\leq m-1$, 
\[
\left(J_{0,n+1}h\right)_{wj}=C_{wj}\left(J_{0,n}h\right)_{w}=C_{wj}E^{1/2}_{w}h=E^{1/2}_{wj}h
\]
by \prettyref{eq:3-1}. This proves the induction step. 
\end{proof}

The direct limit carries natural cylinder projections. At a finite
level, the projection associated with a word $w$ keeps the summands
corresponding to descendants of $w$. These finite level projections
are compatible with the connecting isometries, so they pass to the
direct limit.
\begin{thm}
\label{thm:3-3} For a finite word $w$ and any $n\geq\left|w\right|$,
let $P^{\left(n\right)}_{w}$ be the orthogonal projection of $K_{n}$
onto 
\[
\bigoplus_{\left|u\right|=n,\,u\succeq w}H_{u}
\]
where $u\succeq w$ means $u$ extends $w$. Then 
\begin{equation}
P^{\left(n+1\right)}_{w}J_{n}=J_{n}P^{\left(n\right)}_{w},\qquad n\geq\left|w\right|.\label{eq:3-7}
\end{equation}
Hence there are projections $P_{w}\in B\left(K_{E}\right)$ such that
\begin{equation}
P_{w}U_{n}=U_{n}P^{\left(n\right)}_{w},\qquad n\geq\left|w\right|.\label{eq:3-8}
\end{equation}
Moreover, there is a unique PVM 
\[
P:\mathcal{B}\left(\Omega_{m}\right)\to B\left(K_{E}\right)
\]
such that 
\begin{equation}
P\left(\left[w\right]\right)=P_{w}\label{eq:3-9}
\end{equation}
for every finite word $w$. 
\end{thm}

\begin{proof}
The compatibility relation \prettyref{eq:3-7} follows directly from
the definition of $J_{n}$ in \prettyref{eq:3-4}. Indeed, both sides
keep precisely the components indexed by words extending $w$.

We now define $P_{w}$ on the direct limit. If $x\in K_{r}$ and $n\geq\max\left\{ r,\left|w\right|\right\} $,
set 
\[
P_{w}U_{r}x=U_{n}P^{\left(n\right)}_{w}J_{r,n}x.
\]
This is independent of $n$, by \prettyref{eq:3-7}. It is also independent
of the chosen representative. Indeed, if $U_{r}x=U_{s}y$, then for
some $t\geq\max\left\{ r,s,\left|w\right|\right\} $ we have 
\[
J_{r,t}x=J_{s,t}y.
\]
Applying $P^{\left(t\right)}_{w}$ gives 
\[
P^{\left(t\right)}_{w}J_{r,t}x=P^{\left(t\right)}_{w}J_{s,t}y,
\]
hence 
\[
U_{t}P^{\left(t\right)}_{w}J_{r,t}x=U_{t}P^{\left(t\right)}_{w}J_{s,t}y.
\]
Thus $P_{w}$ is well-defined on $\bigcup_{n\geq0}U_{n}K_{n}$. Since
\[
\left\Vert U_{n}P^{\left(n\right)}_{w}J_{r,n}x\right\Vert \leq\left\Vert U_{n}J_{r,n}x\right\Vert =\left\Vert U_{r}x\right\Vert ,
\]
$P_{w}$ extends to a contraction on $K_{E}$. On the dense subspace
$\bigcup_{n\geq0}U_{n}K_{n}$, the operator is self-adjoint and idempotent,
because each $P^{\left(n\right)}_{w}$ is an orthogonal projection.
Hence the extension is an orthogonal projection. This gives \prettyref{eq:3-8}.

The family $\left\{ P_{w}\right\} $ is compatible with the cylinder
structure. Namely, 
\[
P_{\varnothing}=I_{K_{E}},\qquad P_{w}=\sum^{m-1}_{j=0}P_{wj},
\]
and $P_{u}P_{v}=0$ whenever the words $u$ and $v$ are incompatible.
More generally, $P_{u}P_{v}=P_{z}$ if $\left[u\right]\cap\left[v\right]=\left[z\right]$,
and $P_{u}P_{v}=0$ if the intersection is empty.

By finite additivity, these cylinder values define a finitely additive
PVM on the cylinder algebra. Applying the scalarization and polarization
argument from \prettyref{prop:2-2} gives a POVM 
\[
P:\mathcal{B}\left(\Omega_{m}\right)\to B\left(K_{E}\right)
\]
satisfying \prettyref{eq:3-9}. The cylinder multiplicativity extends
to all Borel sets by the same monotone class argument used in \prettyref{prop:2-4}.
Hence $P$ is projection-valued. 
\end{proof}

We now verify that this PVM dilates the original POVM, and that the
direct limit is minimal.
\begin{thm}
\label{thm:3-4} The triple $\left(K_{E},P,V\right)$ is a minimal
Naimark dilation of $E$. That is, 
\begin{equation}
E\left(B\right)=V^{*}P\left(B\right)V,\qquad B\in\mathcal{B}\left(\Omega_{m}\right),\label{eq:3-10}
\end{equation}
and 
\begin{equation}
K_{E}=\overline{span}\left\{ P\left(B\right)Vh:B\in\mathcal{B}\left(\Omega_{m}\right),\ h\in H\right\} .\label{eq:3-11}
\end{equation}
\end{thm}

\begin{proof}
It is enough first to prove \prettyref{eq:3-10} on cylinders. Let
$w$ be a finite word and set $n=\left|w\right|$. By \prettyref{eq:3-8},
\prettyref{eq:3-9}, and \prettyref{lem:3-2}, 
\[
\left\langle Vh,P\left(\left[w\right]\right)Vk\right\rangle =\left\langle J_{0,n}h,P^{\left(n\right)}_{w}J_{0,n}k\right\rangle =\left\langle E^{1/2}_{w}h,E^{1/2}_{w}k\right\rangle =\left\langle h,E_{w}k\right\rangle .
\]
Thus 
\[
V^{*}P\left(\left[w\right]\right)V=E_{w}.
\]
By uniqueness of the POVM extension from cylinder values, \prettyref{eq:3-10}
holds for all Borel sets.

It remains to prove minimality. Fix $n\geq0$ and a word $w$ with
$\left|w\right|=n$. By \prettyref{eq:3-8} and \prettyref{lem:3-2},
\[
P_{w}Vh=U_{n}\left(0,\ldots,E^{1/2}_{w}h,\ldots,0\right),
\]
where the only nonzero component lies in the $w$-summand of $K_{n}$.
Since $H_{w}=\overline{ran}\left(E^{1/2}_{w}\right)$, the closed
span of vectors of this form is $U_{n}H_{w}$. Summing over all words
of length $n$, we obtain 
\[
U_{n}K_{n}\subseteq\overline{span}\left\{ P\left(B\right)Vh:B\in\mathcal{B}\left(\Omega_{m}\right),\ h\in H\right\} .
\]
Since $K_{E}$ is the closed union of the spaces $U_{n}K_{n}$, \prettyref{eq:3-11}
follows. 
\end{proof}

\begin{rem}
\label{rem:3-5} The construction of $K_{E}$ should be compared with
the usual abstract form of Naimark dilation. The abstract theorem
produces a dilation space and a PVM, but the tree structure need not
remain visible. Here each vertex $w$ carries the range space $H_{w}=\overline{ran}\,(E^{1/2}_{w})$,
and each edge carries the canonical contraction 
\[
C_{wj}:H_{w}\to H_{wj}.
\]
The splitting identity says that the maps from one vertex to all of
its children form an isometry 
\[
J_{w}:H_{w}\to\bigoplus^{m-1}_{j=0}H_{wj}.
\]
Putting together all vertices at the same level gives the isometries
$J_{n}:K_{n}\to K_{n+1}$, and the dilation space is the direct limit
of these level spaces. On this limit, the projection associated with
a cylinder keeps the descendant summands of that cylinder. Compression
back to the initial space recovers the original POVM. Thus the minimal
dilation is realized level by level, in the same local coordinates
that describe the original POVM.
\end{rem}

\section{Extremality and range martingales}\label{sec:4}

We next use the range space dilation to describe extremal POVMs. Extremality
of quantum observables and positive operator-valued measures has been
studied in several settings, including infinite dimensional observables,
commutative observables, generalized quantum measurements, covariant
measurements, and $C^{*}$-extreme points \cite{MR1810813,MR2105199,MR2262678,MR2432912,MR2770378,MR2817979,MR2930514,MR3133476,MR4112227,MR4340930}. 

Let $\mathcal{P}\left(\Omega_{m},H\right)$ denote the convex set
of POVMs on $\Omega_{m}$ with values in $B\left(H\right)$. Recall
that $E\in\mathcal{P}\left(\Omega_{m},H\right)$ is extreme if, whenever
$E=\frac{1}{2}\left(E_{1}+E_{2}\right)$ for $E_{1},E_{2}\in\mathcal{P}\left(\Omega_{m},H\right)$,
it follows that $E_{1}=E_{2}=E$. 

We use the following standard extremality criterion for POVMs, equivalently
Arveson's criterion for unital completely positive maps applied to
the minimal Naimark dilation. 
\begin{prop}
\label{prop:4-1} Let $E$ be a POVM on $\Omega_{m}$, and let $\left(K,P,V\right)$
be a minimal Naimark dilation of $E$. Thus $P:\mathcal{B}\left(\Omega_{m}\right)\to B\left(K\right)$
is a PVM, $V:H\to K$ is an isometry, 
\[
E\left(B\right)=V^{*}P\left(B\right)V,\qquad B\in\mathcal{B}\left(\Omega_{m}\right),
\]
and 
\[
K=\overline{span}\left\{ P\left(B\right)Vh:B\in\mathcal{B}\left(\Omega_{m}\right),\ h\in H\right\} .
\]
Then $E$ is extreme in $\mathcal{P}\left(\Omega_{m},H\right)$ if
and only if 
\[
\left\{ T\in P\left(\mathcal{B}\left(\Omega_{m}\right)\right)'_{\mathrm{sa}}:V^{*}TV=0\right\} =\left\{ 0\right\} .
\]
\end{prop}

\begin{proof}
We give the standard proof sketch, and refer to \cite{MR253059,MR1810813,MR2770378}
for more details.

Suppose first that there is a nonzero self-adjoint $T\in P\left(\mathcal{B}\left(\Omega_{m}\right)\right)'$
with $V^{*}TV=0$. After scaling, assume $\left\Vert T\right\Vert \leq1$.
Then 
\[
E_{\pm}\left(B\right)=V^{*}P\left(B\right)\left(I\pm T\right)V
\]
define POVMs on $\Omega_{m}$. Positivity follows from $I\pm T\geq0$
and the fact that $T$ commutes with $P\left(B\right)$. Normalization
follows from 
\[
E_{\pm}\left(\Omega_{m}\right)=V^{*}\left(I\pm T\right)V=I.
\]
Moreover, $E=\frac{1}{2}\left(E_{+}+E_{-}\right)$. By minimality,
$T\neq0$ makes this decomposition nontrivial. Hence $E$ is not extreme.

Conversely, if $E$ is not extreme, write $E=\frac{1}{2}\left(E_{1}+E_{2}\right)$
with $E_{1}\neq E_{2}$. Since $E_{i}\leq2E$, the Radon-Nikodym theorem
for POVMs gives positive operators $T_{i}\in P\left(\mathcal{B}\left(\Omega_{m}\right)\right)'$
such that $E_{i}\left(B\right)=V^{*}P\left(B\right)T_{i}V$. (See
e.g., \cite{MR932932,MR2014842,MR253059} and also \prettyref{prop:5-1}.)
Then $T=\frac{1}{2}\left(T_{1}-T_{2}\right)$ is self-adjoint, belongs
to the commutant, satisfies $V^{*}TV=0$, and is nonzero by uniqueness
in the Radon-Nikodym theorem. 
\end{proof}

We now return to the range space notation from \prettyref{thm:2-6}
and \prettyref{sec:3}. Thus $E$ is a POVM on $\Omega_{m}$, with
cylinder values 
\[
E_{w}=E\left(\left[w\right]\right),\qquad H_{w}=\overline{ran}\,(E^{1/2}_{w}),
\]
and edge contractions 
\[
C_{wj}:H_{w}\to H_{wj}.
\]

\begin{defn}
\label{def:4-2}A bounded range martingale for $E$ is a family of
self-adjoint operators 
\[
X_{w}\in B\left(H_{w}\right),\qquad w\in\left\{ 0,\ldots,m-1\right\} ^{<\mathbb{N}},
\]
such that 
\[
\sup_{w}\left\Vert X_{w}\right\Vert <\infty
\]
and 
\begin{equation}
X_{w}=\sum^{m-1}_{j=0}C^{*}_{wj}X_{wj}C_{wj}\label{eq:4-1}
\end{equation}
for every finite word $w$. 
\end{defn}

The terminology is motivated by the scalar case. Suppose first that
$E$ is induced by a scalar probability measure $\nu$ on $\Omega_{m}$.
Then each nonzero range space $H_{w}$ is canonically $\mathbb{C}$,
and, whenever $\nu\left(\left[w\right]\right)>0$, the local splitting
operators are multiplication by the conditional probabilities 
\[
p_{w}\left(j\right)=\frac{\nu\left(\left[wj\right]\right)}{\nu\left(\left[w\right]\right)}.
\]
The edge map $C_{wj}:H_{w}\to H_{wj}$ is then multiplication by $\sqrt{p_{w}\left(j\right)}$.
If $\nu\left(\left[wj\right]\right)=0$, the child range space is
zero, and the corresponding term contributes nothing. A choice of
operators $X_{w}\in B\left(H_{w}\right)$ is therefore a choice of
scalars $x_{w}$ indexed by the tree vertices. Equivalently, it determines
a scalar process $\left(M_{n}\right)$ by setting 
\[
M_{n}\left(\omega\right)=x_{w}
\]
for $\omega\in\left[w\right]$, $\left|w\right|=n$. With this identification,
\prettyref{eq:4-1} becomes 
\[
x_{w}=\sum^{m-1}_{j=0}p_{w}\left(j\right)x_{wj},
\]
which is the condition 
\[
M_{n}=\mathbb{E}\left(M_{n+1}\mid\mathcal{F}_{n}\right)
\]
where $\mathcal{F}_{n}$ is the finite $\sigma$-algebra generated
by the cylinders of length $n$. Thus the value at a parent cylinder
is the conditional average of the values on its child cylinders. \prettyref{def:4-2}
is the same martingale relation, but with scalar cylinder values replaced
by operators on the range spaces and with the conditional probabilities
replaced by the edge contractions $C_{wj}$.
\begin{cor}
\label{cor:4-3} Let $X=\left\{ X_{w}\right\} $ be a bounded range
martingale for $E$, and let $h\in H$ be a unit vector. Define the
scalar probability measure 
\[
\nu_{h}\left(B\right)=\left\langle h,E\left(B\right)h\right\rangle ,\qquad B\in\mathcal{B}\left(\Omega_{m}\right).
\]
For each word $w$ with $\nu_{h}\left(\left[w\right]\right)>0$, set
\[
x^{h}_{w}=\frac{\left\langle E^{1/2}_{w}h,X_{w}E^{1/2}_{w}h\right\rangle }{\left\Vert E^{1/2}_{w}h\right\Vert ^{2}}.
\]
On words with $\nu_{h}\left(\left[w\right]\right)=0$, define $x^{h}_{w}=0$.
Let $\mathcal{F}_{n}$ be the finite $\sigma$-algebra generated by
the cylinders of length $n$, and define 
\[
M^{h}_{n}\left(\omega\right)=x^{h}_{w},\qquad\omega\in\left[w\right],\quad\left|w\right|=n.
\]
Then $\left\{ M^{h}_{n}\right\} _{n\geq0}$ is a bounded scalar martingale
with respect to $\left\{ \mathcal{F}_{n}\right\} _{n\geq0}$ and $\nu_{h}$. 
\end{cor}

\begin{proof}
It is enough to check the martingale relation on each cylinder $\left[w\right]$
with $\nu_{h}\left(\left[w\right]\right)>0$. By \prettyref{eq:4-1},
\[
\begin{aligned}\left\langle E^{1/2}_{w}h,X_{w}E^{1/2}_{w}h\right\rangle  & =\sum^{m-1}_{j=0}\left\langle C_{wj}E^{1/2}_{w}h,X_{wj}C_{wj}E^{1/2}_{w}h\right\rangle \\
 & =\sum^{m-1}_{j=0}\left\langle E^{1/2}_{wj}h,X_{wj}E^{1/2}_{wj}h\right\rangle .
\end{aligned}
\]
Also 
\[
\nu_{h}\left(\left[w\right]\right)=\sum^{m-1}_{j=0}\nu_{h}\left(\left[wj\right]\right).
\]
Therefore 
\[
x^{h}_{w}=\sum^{m-1}_{j=0}\frac{\nu_{h}\left(\left[wj\right]\right)}{\nu_{h}\left(\left[w\right]\right)}x^{h}_{wj}.
\]
This is exactly the identity 
\[
M^{h}_{n}=\mathbb{E}_{\nu_{h}}\left(M^{h}_{n+1}\mid\mathcal{F}_{n}\right).
\]
Boundedness follows from 
\[
\left|x^{h}_{w}\right|\leq\sup_{u}\left\Vert X_{u}\right\Vert .
\]
\end{proof}

The martingale relation also reflects the commutant of the dilating
PVM in the range space coordinates. An operator commuting with the
cylinder projections is block diagonal at each finite level 
\[
K_{n}=\bigoplus_{\left|w\right|=n}H_{w}.
\]
Compatibility with the direct limit maps then forces the block at
a parent vertex to be obtained from the blocks at its children through
\prettyref{eq:4-1}. The next proposition shows that this gives all
operators in the commutant.
\begin{thm}
\label{thm:4-4} Let $\left(K_{E},P,V\right)$ be the range space
dilation from \prettyref{sec:3}. Self-adjoint operators $T\in B\left(K_{E}\right)$
satisfying 
\[
TP\left(B\right)=P\left(B\right)T,\qquad B\in\mathcal{B}\left(\Omega_{m}\right),
\]
are in one-to-one correspondence with bounded range martingales $\left\{ X_{w}\right\} $.
The correspondence is given by 
\begin{equation}
U^{*}_{n}TU_{n}=\bigoplus_{\left|w\right|=n}X_{w},\qquad n\geq0.\label{eq:4-2}
\end{equation}
\end{thm}

\begin{proof}
Suppose first that $T=T^{*}$ commutes with $P\left(B\right)$ for
every Borel set $B$. Set 
\[
Y_{n}=U^{*}_{n}TU_{n}\in B\left(K_{n}\right).
\]
If $\left|w\right|=n$, then 
\[
P\left(\left[w\right]\right)U_{n}=U_{n}P^{\left(n\right)}_{w}
\]
by \prettyref{eq:3-8}. Since $T$ commutes with $P\left(\left[w\right]\right)$,
the operator $Y_{n}$ commutes with $P^{\left(n\right)}_{w}$ for
every word $w$ of length $n$. Hence $Y_{n}$ is block diagonal with
respect to 
\[
K_{n}=\bigoplus_{\left|w\right|=n}H_{w}.
\]
Thus there are self-adjoint operators $X_{w}\in B\left(H_{w}\right)$
such that 
\[
Y_{n}=\bigoplus_{\left|w\right|=n}X_{w}.
\]
Clearly 
\[
\sup_{w}\left\Vert X_{w}\right\Vert \leq\left\Vert T\right\Vert .
\]

It remains to check \prettyref{eq:4-1}. Since $U_{n+1}J_{n}=U_{n}$,
we have 
\[
Y_{n}=J^{*}_{n}Y_{n+1}J_{n}.
\]
Using the block form of $Y_{n}$ and $Y_{n+1}$, together with the
definition of $J_{n}$ in \prettyref{eq:3-4}, this identity gives
\[
X_{w}=\sum^{m-1}_{j=0}C^{*}_{wj}X_{wj}C_{wj}
\]
for every word $w$ of length $n$. Hence $\left\{ X_{w}\right\} $
is a bounded range martingale.

Conversely, suppose that $\left\{ X_{w}\right\} $ is a bounded range
martingale. Define 
\[
Y_{n}=\bigoplus_{\left|w\right|=n}X_{w}\in B\left(K_{n}\right).
\]
Then \prettyref{eq:4-1} is equivalent to 
\[
Y_{n}=J^{*}_{n}Y_{n+1}J_{n},\qquad n\geq0.
\]
For vectors $x\in K_{r}$ and $y\in K_{s}$, choose $n\geq\max\left\{ r,s\right\} $
and set 
\[
\gamma\left(U_{r}x,U_{s}y\right)=\left\langle J_{r,n}x,Y_{n}J_{s,n}y\right\rangle .
\]
The preceding compatibility relation shows that this definition is
independent of $n$ and of the representatives. If 
\[
M=\sup_{w}\left\Vert X_{w}\right\Vert ,
\]
then 
\[
\left|\gamma\left(U_{r}x,U_{s}y\right)\right|\leq M\left\Vert x\right\Vert \left\Vert y\right\Vert .
\]
Thus $\gamma$ extends to a bounded sesquilinear form on $K_{E}$.
Hence there is a self-adjoint operator $T\in B\left(K_{E}\right)$
such that 
\[
\gamma\left(\xi,\eta\right)=\left\langle \xi,T\eta\right\rangle ,\qquad\xi,\eta\in K_{E}.
\]
By construction, 
\[
U^{*}_{n}TU_{n}=Y_{n}
\]
for every $n$.

It remains to show that $T$ commutes with $P$. Let $w$ be a finite
word. If $n\geq\left|w\right|$, then $Y_{n}$ commutes with $P^{\left(n\right)}_{w}$,
since $Y_{n}$ is block diagonal at level $n$ and $P^{\left(n\right)}_{w}$
is the projection onto a sum of level $n$ blocks. It follows, by
evaluating matrix coefficients on the dense subspace $\bigcup_{n\geq0}U_{n}K_{n}$,
that 
\[
TP\left(\left[w\right]\right)=P\left(\left[w\right]\right)T.
\]
Since the cylinder sets generate $\mathcal{B}\left(\Omega_{m}\right)$,
a monotone class argument gives 
\[
TP\left(B\right)=P\left(B\right)T
\]
for every Borel set $B\subseteq\Omega_{m}$.

The two constructions are inverse to each other, so the correspondence
is one-to-one. 
\end{proof}

\begin{thm}
\label{thm:4-5} The POVM $E$ is extreme in $\mathcal{P}\left(\Omega_{m},H\right)$
if and only if the only bounded range martingale $\left\{ X_{w}\right\} $
satisfying 
\[
X_{\varnothing}=0
\]
is the zero family. 
\end{thm}

\begin{proof}
By \prettyref{thm:4-4}, bounded range martingales correspond to self-adjoint
operators in the commutant of the dilation PVM $P$. Under this correspondence,
\[
X_{\varnothing}=U^{*}_{0}TU_{0}=V^{*}TV,
\]
since $U_{0}=V$ and $K_{0}=H$. Therefore \prettyref{prop:4-1} gives
the result. 
\end{proof}

\begin{rem}
\label{rem:4-6} In the scalar case, \prettyref{thm:4-5} recovers
the familiar fact that the extreme probability measures on $\Omega_{m}$
are the point masses. Indeed, if the measure is not a point mass,
then some cylinder has measure strictly between $0$ and $1$, and
the centered conditional expectations of that cylinder give a nonzero
bounded martingale with initial value zero. If the measure is a point
mass, every nonzero range space lies on one branch, and the recursion
forces all values to agree with the initial value. 

The same criterion also recovers the classical fact that PVMs are
extreme. If $E$ is projection-valued, then $H_{w}=\oplus^{m-1}_{j=0}H_{wj}$,
and \prettyref{eq:4-1} says that $X_{w}$ is the block diagonal operator
with diagonal blocks $X_{w0},\ldots,X_{w,m-1}$. Thus $X_{\varnothing}=0$
forces all first-level blocks to vanish, and induction gives $X_{w}=0$
for every finite word $w$.
\end{rem}

\section{Domination}\label{sec:5}

The same range martingales describe positive operator-valued measures
dominated by $E$. The background is the Radon-Nikodym theory for
completely positive maps, quantum operations, and quantum instruments
\cite{MR932932,MR1608473,MR2014842}. Here that standard domination
theorem is written in the range space coordinates of \prettyref{sec:3}. 

Let $\mathcal{M}_{+}\left(\Omega_{m},H\right)$ denote the cone of
finite positive operator-valued measures on $\Omega_{m}$ with values
in $B\left(H\right)_{+}$. These measures are not assumed to be normalized.
Thus 
\[
\mathcal{P}\left(\Omega_{m},H\right)=\left\{ F\in\mathcal{M}_{+}\left(\Omega_{m},H\right):F\left(\Omega_{m}\right)=I\right\} .
\]
As before, each element of $\mathcal{M}_{+}\left(\Omega_{m},H\right)$
is determined by its cylinder values, by the scalarization and polarization
argument used in the proof of \prettyref{prop:2-2}.

If $F\in\mathcal{M}_{+}\left(\Omega_{m},H\right)$ and $c>0$, we
write 
\[
F\leq cE
\]
when 
\[
0\leq F\left(B\right)\leq cE\left(B\right)
\]
for every Borel set $B\subseteq\Omega_{m}$.

We use the following standard Radon-Nikodym theorem for dominated
operator-valued measures (see e.g., \cite{MR932932,MR2014842,MR253059}).
\begin{prop}
\label{prop:5-1} Let $\left(K,P,V\right)$ be a minimal Naimark dilation
of $E$, and let $c>0$. If $F\in\mathcal{M}_{+}\left(\Omega_{m},H\right)$
satisfies $F\leq cE$, then there is a unique positive operator $T\in B\left(K\right)$
such that 
\[
0\leq T\leq cI_{K},
\]
\[
TP\left(B\right)=P\left(B\right)T,\qquad B\in\mathcal{B}\left(\Omega_{m}\right),
\]
and 
\[
F\left(B\right)=V^{*}P\left(B\right)TV,\qquad B\in\mathcal{B}\left(\Omega_{m}\right).
\]
Conversely, every such operator $T$ defines an element $F\in\mathcal{M}_{+}\left(\Omega_{m},H\right)$
satisfying $F\leq cE$ by this formula.
\end{prop}

We now translate this standard result into the range spaces of $E$.
A positive range martingale means a bounded range martingale $\left\{ X_{w}\right\} $
from \prettyref{def:4-2} such that $X_{w}\geq0$ for every finite
word $w$.
\begin{thm}
\label{thm:5-2} Let $c>0$. The assignment 
\[
F\mapsto\left\{ X_{w}\right\} 
\]
is a one-to-one correspondence between measures $F\in\mathcal{M}_{+}\left(\Omega_{m},H\right)$
satisfying $F\leq cE$ and positive range martingales $\left\{ X_{w}\right\} $
satisfying 
\begin{equation}
0\leq X_{w}\leq cI_{H_{w}}\label{eq:5-1}
\end{equation}
for every finite word $w$. The correspondence is characterized by
\begin{equation}
F\left(\left[w\right]\right)=E^{1/2}_{w}X_{w}E^{1/2}_{w}\label{eq:5-2}
\end{equation}
for every finite word $w$. In \prettyref{eq:5-2}, $X_{w}$ is extended
by zero on $H^{\perp}_{w}$. 
\end{thm}

\begin{proof}
Suppose first that $F\leq cE$. By \prettyref{prop:5-1}, there is
a unique positive operator $T\in B\left(K_{E}\right)$ such that 
\[
0\leq T\leq cI_{K_{E}},
\]
$T$ commutes with $P\left(B\right)$ for every Borel set $B$, and
\[
F\left(B\right)=V^{*}P\left(B\right)TV,\qquad B\in\mathcal{B}\left(\Omega_{m}\right).
\]
By \prettyref{thm:4-4}, $T$ determines a bounded range martingale
$\left\{ X_{w}\right\} $ through 
\[
U^{*}_{n}TU_{n}=\bigoplus_{\left|w\right|=n}X_{w},\qquad n\geq0.
\]
Since $0\leq T\leq cI_{K_{E}}$, we have 
\[
0\leq U^{*}_{n}TU_{n}\leq cI_{K_{n}}.
\]
Therefore each block satisfies 
\[
0\leq X_{w}\leq cI_{H_{w}}.
\]

It remains to prove \prettyref{eq:5-2}. Let $w$ be a word of length
$n$. For $h,k\in H$, using \prettyref{eq:3-8}, \prettyref{lem:3-2},
and the block form of $U^{*}_{n}TU_{n}$, we get 
\[
\begin{aligned}\left\langle h,F\left(\left[w\right]\right)k\right\rangle  & =\left\langle Vh,P\left(\left[w\right]\right)TVk\right\rangle \\
 & =\left\langle P\left(\left[w\right]\right)Vh,TVk\right\rangle \\
 & =\left\langle U_{n}P^{\left(n\right)}_{w}J_{0,n}h,TU_{n}J_{0,n}k\right\rangle \\
 & =\left\langle P^{\left(n\right)}_{w}J_{0,n}h,U^{*}_{n}TU_{n}J_{0,n}k\right\rangle \\
 & =\left\langle E^{1/2}_{w}h,X_{w}E^{1/2}_{w}k\right\rangle \\
 & =\left\langle h,E^{1/2}_{w}X_{w}E^{1/2}_{w}k\right\rangle .
\end{aligned}
\]
Thus \prettyref{eq:5-2} holds.

Conversely, suppose that $\left\{ X_{w}\right\} $ is a positive range
martingale satisfying \prettyref{eq:5-1}. By \prettyref{thm:4-4},
there is a self-adjoint operator $T\in B\left(K_{E}\right)$ commuting
with $P\left(B\right)$ for every Borel set $B$, and satisfying 
\[
U^{*}_{n}TU_{n}=\bigoplus_{\left|w\right|=n}X_{w},\qquad n\geq0.
\]
We claim that 
\[
0\leq T\leq cI_{K_{E}}.
\]
Indeed, let $\xi=U_{r}x$ with $x\in K_{r}$. Choose $n\geq r$. Then
\[
\left\langle \xi,T\xi\right\rangle =\left\langle J_{r,n}x,\left(\bigoplus_{\left|w\right|=n}X_{w}\right)J_{r,n}x\right\rangle \geq0.
\]
Similarly, 
\[
\left\langle \xi,\left(cI_{K_{E}}-T\right)\xi\right\rangle =\left\langle J_{r,n}x,\left(\bigoplus_{\left|w\right|=n}\left(cI_{H_{w}}-X_{w}\right)\right)J_{r,n}x\right\rangle \geq0.
\]
Since $\bigcup_{r\geq0}U_{r}K_{r}$ is dense in $K_{E}$, the claim
follows.

Define 
\[
F\left(B\right)=V^{*}P\left(B\right)TV,\qquad B\in\mathcal{B}\left(\Omega_{m}\right).
\]
Since $T$ is positive and commutes with $P\left(B\right)$, $F$
is a finite positive operator-valued measure. Also, 
\[
0\leq F\left(B\right)\leq cV^{*}P\left(B\right)V=cE\left(B\right),
\]
so $F\leq cE$.

The same calculation as in the first half of the proof gives 
\[
F\left(\left[w\right]\right)=E^{1/2}_{w}X_{w}E^{1/2}_{w}
\]
for every finite word $w$.

Finally, the uniqueness of $\left\{ X_{w}\right\} $ follows from
\prettyref{eq:5-2}. Indeed, if another positive family $\left\{ Y_{w}\right\} $
satisfies the same cylinder formula, then 
\[
E^{1/2}_{w}X_{w}E^{1/2}_{w}=E^{1/2}_{w}Y_{w}E^{1/2}_{w}
\]
for every $w$. Since $X_{w}$ and $Y_{w}$ act on $H_{w}=\overline{ran}\,(E^{1/2}_{w})$,
the density argument from \prettyref{lem:2-1} gives $X_{w}=Y_{w}$. 
\end{proof}

The root value of the martingale is the total mass of the dominated
measure.
\begin{cor}
\label{cor:5-3} Under the correspondence in \prettyref{thm:5-2},
\[
X_{\varnothing}=F\left(\Omega_{m}\right).
\]
In particular, the measures $F\in\mathcal{M}_{+}\left(\Omega_{m},H\right)$
satisfying $F\leq cE$ and having fixed total mass $A\in B\left(H\right)_{+}$
are in one-to-one correspondence with the positive range martingales
satisfying 
\[
0\leq X_{w}\leq cI_{H_{w}}
\]
for every $w$, and 
\[
X_{\varnothing}=A.
\]
\end{cor}

\begin{proof}
Since $E_{\varnothing}=I$ and $H_{\varnothing}=H$, \prettyref{eq:5-2}
with $w=\varnothing$ gives 
\[
F\left(\Omega_{m}\right)=X_{\varnothing}.
\]
The remaining statement follows from \prettyref{thm:5-2}. 
\end{proof}

The preceding theorem also describes convex decompositions of $E$.
This gives another way to view the extremality criterion from \prettyref{sec:4}.
\begin{cor}
\label{cor:5-4} Convex decompositions 
\[
E=\frac{1}{2}\left(F+G\right),
\]
where $F$ and $G$ are POVMs on $\Omega_{m}$ with values in $B\left(H\right)$,
are in one-to-one correspondence with positive range martingales $\left\{ X_{w}\right\} $
satisfying 
\begin{equation}
0\leq X_{w}\leq2I_{H_{w}}\label{eq:5-3}
\end{equation}
for every finite word $w$, and 
\begin{equation}
X_{\varnothing}=I.\label{eq:5-4}
\end{equation}
The correspondence is given on cylinders by 
\begin{equation}
F\left(\left[w\right]\right)=E^{1/2}_{w}X_{w}E^{1/2}_{w}\label{eq:5-5}
\end{equation}
and 
\begin{equation}
G\left(\left[w\right]\right)=E^{1/2}_{w}\left(2I_{H_{w}}-X_{w}\right)E^{1/2}_{w}.\label{eq:5-6}
\end{equation}
\end{cor}

\begin{proof}
Suppose first that 
\[
E=\frac{1}{2}\left(F+G\right)
\]
with $F$ and $G$ POVMs. Then 
\[
F\leq2E.
\]
By \prettyref{thm:5-2}, $F$ corresponds to a positive range martingale
$\left\{ X_{w}\right\} $ satisfying \prettyref{eq:5-3}. Since $F\left(\Omega_{m}\right)=I$,
\prettyref{cor:5-3} gives \prettyref{eq:5-4}. Formula \prettyref{eq:5-5}
is just \prettyref{eq:5-2}. Since $F+G=2E$, \prettyref{eq:5-6}
follows from 
\[
2E_{w}=E^{1/2}_{w}\left(2I_{H_{w}}\right)E^{1/2}_{w}.
\]

Conversely, suppose that $\left\{ X_{w}\right\} $ is a positive range
martingale satisfying \prettyref{eq:5-3} and \prettyref{eq:5-4}.
By \prettyref{thm:5-2}, it defines a measure $F\in\mathcal{M}_{+}\left(\Omega_{m},H\right)$
satisfying $F\leq2E$ through \prettyref{eq:5-5}. By \prettyref{cor:5-3},
$F\left(\Omega_{m}\right)=I$, so $F$ is a POVM.

The family $\left\{ 2I_{H_{w}}-X_{w}\right\} $ is also a positive
range martingale. Indeed, 
\[
2I_{H_{w}}-X_{w}=\sum^{m-1}_{j=0}C^{*}_{wj}\left(2I_{H_{wj}}-X_{wj}\right)C_{wj},
\]
because 
\[
I_{H_{w}}=\sum^{m-1}_{j=0}C^{*}_{wj}C_{wj}.
\]
By \prettyref{thm:5-2}, it defines a measure $G\in\mathcal{M}_{+}\left(\Omega_{m},H\right)$
such that $G\leq2E$ by \prettyref{eq:5-6}. Since 
\[
G\left(\Omega_{m}\right)=2I-X_{\varnothing}=I,
\]
the measure $G$ is a POVM. Adding \prettyref{eq:5-5} and \prettyref{eq:5-6}
gives 
\[
F\left(\left[w\right]\right)+G\left(\left[w\right]\right)=2E_{w}
\]
for every finite word $w$. Since finite positive operator-valued
measures on $\Omega_{m}$ are determined by their cylinder values,
\[
F+G=2E.
\]
Thus 
\[
E=\frac{1}{2}\left(F+G\right).
\]
The two constructions are inverse to each other. 
\end{proof}

\begin{rem}
\label{rem:5-5} In this language, $E$ is extreme if and only if
the only positive range martingale satisfying 
\[
0\leq X_{w}\leq2I_{H_{w}},\qquad X_{\varnothing}=I,
\]
is the identity family 
\[
X_{w}=I_{H_{w}},\qquad w\in\left\{ 0,\ldots,m-1\right\} ^{<\mathbb{N}}.
\]
This is equivalent to \prettyref{thm:4-5}. Indeed, if such a positive
range martingale $\left\{ X_{w}\right\} $ is not the identity family,
then 
\[
Y_{w}=X_{w}-I_{H_{w}}
\]
is a nonzero bounded self-adjoint range martingale with $Y_{\varnothing}=0$.
Conversely, if $\left\{ Y_{w}\right\} $ is a nonzero bounded self-adjoint
range martingale with $Y_{\varnothing}=0$, then for sufficiently
small $t>0$, 
\[
X_{w}=I_{H_{w}}+tY_{w}
\]
satisfies 
\[
0\leq X_{w}\leq2I_{H_{w}},\qquad X_{\varnothing}=I,
\]
and is not the identity family. 
\end{rem}

\section{Doob transform}\label{sec:6}

The positive range martingales from \prettyref{sec:5} give a change
of measure calculus for tree POVMs. In the scalar case this is the
classical change of measure by a positive martingale on a filtered
probability space, and on a Markov tree it gives the usual Doob transform
of the transition probabilities \cite{MR109961,MR133152}. In the
present setting the same idea changes not only the cylinder values,
but also the local splitting operators. The operator-valued side is
related to Bayes-type and Radon-Nikodym formulas for quantum random
variables, completely positive maps, and quantum operations \cite{MR2953266,MR932932,MR2014842}.

Throughout this section $E$ is a POVM on $\Omega_{m}$, with cylinder
values $E_{w}$, range spaces 
\[
H_{w}=\overline{ran}\,(E^{1/2}_{w}),
\]
and edge contractions $C_{wj}:H_{w}\to H_{wj}$ as in \prettyref{sec:3}.
\begin{defn}
A positive range martingale $X=\left\{ X_{w}\right\} $ for $E$ is
called strictly positive if there are constants $0<a\leq b<\infty$
such that 
\[
aI_{H_{w}}\leq X_{w}\leq bI_{H_{w}}
\]
for every finite word $w$. It is called normalized if 
\[
X_{\varnothing}=I_{H}.
\]
\end{defn}

Indeed, strictly positive normalized range martingales for $E$ are
in one-to-one correspondence with POVMs boundedly equivalent to $E$.
\begin{prop}
\label{prop:6-2} There is a one-to-one correspondence 
\begin{multline*}
\left\{ \begin{array}{c}
\text{strictly positive normalized}\\
\text{range martingales }X=\left\{ X_{w}\right\} _{w}\text{ for }E
\end{array}\right\} \\
\Longleftrightarrow\left\{ \begin{array}{c}
\text{POVMs }F\text{ on }\Omega_{m}\text{ such that}\\
aE\leq F\leq bE\text{ for some }0<a\leq b<\infty
\end{array}\right\} .
\end{multline*}
The correspondence is given on cylinders by 
\begin{equation}
F\left(\left[w\right]\right)=E^{1/2}_{w}X_{w}E^{1/2}_{w}.\label{eq:6-1-1}
\end{equation}
 
\end{prop}

\begin{proof}
Let $X=\left\{ X_{w}\right\} $ be a strictly positive normalized
range martingale. Since $X$ is bounded and positive, \prettyref{thm:5-2}
gives a finite positive operator-valued measure $E^{X}$ whose cylinder
values are 
\begin{equation}
E^{X}_{w}=E^{1/2}_{w}X_{w}E^{1/2}_{w}.\label{eq:6-1}
\end{equation}
Since $X_{\varnothing}=I_{H}$, \prettyref{cor:5-3} gives $E^{X}\left(\Omega_{m}\right)=I_{H}$.
Thus $E^{X}$ is a POVM on $\Omega_{m}$. Moreover, 
\[
aE\leq E^{X}\leq bE.
\]
Indeed, the families 
\[
\left\{ X_{w}-aI_{H_{w}}\right\} _{w}\qquad\text{and}\qquad\left\{ bI_{H_{w}}-X_{w}\right\} _{w}
\]
are positive range martingales, since the identity family 
\[
\left\{ I_{H_{w}}\right\} _{w}
\]
is a range martingale. By \prettyref{thm:5-2}, they define the positive
measures 
\[
E^{X}-aE\qquad\text{and}\qquad bE-E^{X},
\]
respectively. Hence $aE\leq E^{X}\leq bE$. At the root this says
\[
aI_{H}\leq E^{X}\left(\Omega_{m}\right)\leq bI_{H}.
\]
Since $X_{\varnothing}=I_{H}$, we have 
\[
E^{X}\left(\Omega_{m}\right)=I_{H},
\]
so the root inequality is just 
\[
aI_{H}\leq I_{H}\leq bI_{H}.
\]

Conversely, suppose $F$ is a POVM such that 
\[
aE\leq F\leq bE
\]
for some $0<a\leq b<\infty$. Since $F\leq bE$, \prettyref{thm:5-2}
gives a unique positive range martingale $X=\left\{ X_{w}\right\} $
such that 
\[
F\left(\left[w\right]\right)=E^{1/2}_{w}X_{w}E^{1/2}_{w}.
\]
The bound $X_{w}\leq bI_{H_{w}}$ follows from \prettyref{thm:5-2}.
To get the lower bound, apply the same theorem to the positive measure
$F-aE$. Its cylinder values are 
\[
F\left(\left[w\right]\right)-aE_{w}=E^{1/2}_{w}\left(X_{w}-aI_{H_{w}}\right)E^{1/2}_{w}.
\]
Hence 
\[
X_{w}-aI_{H_{w}}\geq0.
\]
Thus 
\[
aI_{H_{w}}\leq X_{w}\leq bI_{H_{w}}
\]
for every $w$. Since $F\left(\Omega_{m}\right)=I$, we also have
\[
X_{\varnothing}=I_{H}.
\]
Thus $X$ is strictly positive and normalized. The two constructions
are inverse to each other by the uniqueness statement in \prettyref{thm:5-2}.
\end{proof}

We now compute how the local tree coordinates change under this operation.
For the transformed POVM $E^{X}$, the canonical range space at $w$
would be 
\[
H^{X}_{w}=\overline{ran}\,((E^{X}_{w})^{1/2}).
\]
This space is naturally attached to $E^{X}_{w}$, but it is not literally
the original space $H_{w}$. We therefore use an equivalent factorization
of $E^{X}_{w}$ whose target space is $H_{w}$. Set 
\[
S^{X}_{w}=X^{1/2}_{w}E^{1/2}_{w}:H\to H_{w}.
\]
Then 
\[
E^{X}_{w}=((E^{X}_{w})^{1/2})^{*}(E^{X}_{w})^{1/2}=\left(S^{X}_{w}\right)^{*}S^{X}_{w}.
\]
Moreover, 
\[
ran\left(S^{X}_{w}\right)=X^{1/2}_{w}ran\,(E^{1/2}_{w})
\]
is dense in $H_{w}$, because $ran\,(E^{1/2}_{w})$ is dense in $H_{w}$
and $X^{1/2}_{w}$ is invertible on $H_{w}$. Thus both factorizations
have dense range in their target spaces.

It follows that there is a unique unitary 
\[
U^{X}_{w}:H^{X}_{w}\to H_{w}
\]
such that 
\[
U^{X}_{w}\left(E^{X}_{w}\right)^{1/2}h=S^{X}_{w}h
\]
for every $h\in H$. Hence we may compute the transformed edge contractions
on the original spaces $H_{w}$, using the factors $S^{X}_{w}$.
\begin{thm}
\label{thm:6-3} Let $X=\left\{ X_{w}\right\} $ be a strictly positive
normalized range martingale for $E$. In the factor space coordinates
$S^{X}_{w}=X^{1/2}_{w}E^{1/2}_{w}$, the edge contractions of $E^{X}$
are 
\begin{equation}
C^{X}_{wj}=X^{1/2}_{wj}C_{wj}X^{-1/2}_{w}.\label{eq:6-2}
\end{equation}
Consequently, the splitting operators of $E^{X}$ in these coordinates
are 
\begin{equation}
A^{X,\left(j\right)}_{w}=X^{-1/2}_{w}C^{*}_{wj}X_{wj}C_{wj}X^{-1/2}_{w}.\label{eq:6-3}
\end{equation}
\end{thm}

\begin{proof}
In the factor space coordinates, the edge contraction for $E^{X}$
is determined on the dense subspace $ran\left(S^{X}_{w}\right)\subseteq H_{w}$
by 
\[
C^{X}_{wj}S^{X}_{w}h=S^{X}_{wj}h,\qquad h\in H.
\]
Using $S^{X}_{w}=X^{1/2}_{w}E^{1/2}_{w}$, this becomes 
\[
C^{X}_{wj}X^{1/2}_{w}E^{1/2}_{w}h=X^{1/2}_{wj}E^{1/2}_{wj}h.
\]
By the defining relation for the original edge contraction, 
\[
C_{wj}E^{1/2}_{w}h=E^{1/2}_{wj}h.
\]
Hence 
\[
C^{X}_{wj}X^{1/2}_{w}E^{1/2}_{w}h=X^{1/2}_{wj}C_{wj}E^{1/2}_{w}h.
\]
Since $ran\,(E^{1/2}_{w})$ is dense in $H_{w}$, this gives 
\[
C^{X}_{wj}X^{1/2}_{w}x=X^{1/2}_{wj}C_{wj}x,\qquad x\in H_{w}.
\]
Thus 
\[
C^{X}_{wj}=X^{1/2}_{wj}C_{wj}X^{-1/2}_{w}.
\]
Taking adjoint products gives \prettyref{eq:6-3}. 
\end{proof}

The transformed splitting operators still form a positive decomposition
of the identity. This is the local form of the martingale identity.
\begin{cor}
\label{cor:6-4} For every finite word $w$, 
\[
\sum^{m-1}_{j=0}A^{X,\left(j\right)}_{w}=I_{H_{w}}.
\]
\end{cor}

\begin{proof}
Using \prettyref{eq:6-3} and the range martingale identity for $X$,
\[
\begin{aligned}\sum^{m-1}_{j=0}A^{X,\left(j\right)}_{w} & =X^{-1/2}_{w}\left(\sum^{m-1}_{j=0}C^{*}_{wj}X_{wj}C_{wj}\right)X^{-1/2}_{w}\\
 & =X^{-1/2}_{w}X_{w}X^{-1/2}_{w}=I_{H_{w}}.
\end{aligned}
\]
\end{proof}

The transforms compose by a noncommutative product of martingales.
Suppose $X$ is a strictly positive normalized range martingale for
$E$. Regard $E^{X}$ in the factor space coordinates from \prettyref{thm:6-3}.
If $Y=\left\{ Y_{w}\right\} $ is a strictly positive normalized range
martingale for $E^{X}$ in these coordinates, define 
\begin{equation}
Z_{w}=X^{1/2}_{w}Y_{w}X^{1/2}_{w}.\label{eq:6-4}
\end{equation}

\begin{prop}
\label{prop:6-5} The family $Z=\left\{ Z_{w}\right\} $ is a strictly
positive normalized range martingale for $E$, and 
\[
\left(E^{X}\right)^{Y}=E^{Z}.
\]
\end{prop}

\begin{proof}
The normalization is immediate: 
\[
Z_{\varnothing}=X^{1/2}_{\varnothing}Y_{\varnothing}X^{1/2}_{\varnothing}=I_{H}.
\]
Strict positivity follows from the corresponding bounds for $X$ and
$Y$.

It remains to verify the martingale identity. Since $Y$ is a range
martingale for $E^{X}$ in the factor space coordinates, 
\[
Y_{w}=\sum^{m-1}_{j=0}\left(C^{X}_{wj}\right)^{*}Y_{wj}C^{X}_{wj}.
\]
Using \prettyref{eq:6-2}, we get 
\[
Y_{w}=X^{-1/2}_{w}\left(\sum^{m-1}_{j=0}C^{*}_{wj}X^{1/2}_{wj}Y_{wj}X^{1/2}_{wj}C_{wj}\right)X^{-1/2}_{w}.
\]
Multiplying on the left and right by $X^{1/2}_{w}$ gives 
\[
Z_{w}=\sum^{m-1}_{j=0}C^{*}_{wj}Z_{wj}C_{wj}.
\]
Thus $Z$ is a range martingale for $E$.

Finally, for every word $w$, 
\[
\left(E^{X}\right)^{Y}_{w}=\left(S^{X}_{w}\right)^{*}Y_{w}S^{X}_{w}=E^{1/2}_{w}X^{1/2}_{w}Y_{w}X^{1/2}_{w}E^{1/2}_{w}=E^{1/2}_{w}Z_{w}E^{1/2}_{w}=E^{Z}_{w}.
\]
Since POVMs on $\Omega_{m}$ are determined by their cylinder values,
\[
\left(E^{X}\right)^{Y}=E^{Z}.
\]
\end{proof}

The inverse transform is explicit.
\begin{cor}
\label{cor:6-6} Let $X$ be a strictly positive normalized range
martingale for $E$. In the $E^{X}$-coordinates from \prettyref{thm:6-3},
the family 
\[
X^{-1}=\left\{ X^{-1}_{w}\right\} 
\]
is a strictly positive normalized range martingale for $E^{X}$, and
\[
\left(E^{X}\right)^{X^{-1}}=E.
\]
\end{cor}

\begin{proof}
Since $X$ is strictly positive and normalized, the family 
\[
X^{-1}=\left\{ X^{-1}_{w}\right\} _{w}
\]
is also strictly positive and satisfies $X^{-1}_{\varnothing}=I_{H}$.
Using \prettyref{eq:6-2}, 
\[
\begin{aligned}\sum^{m-1}_{j=0}\left(C^{X}_{wj}\right)^{*}X^{-1}_{wj}C^{X}_{wj} & =\sum^{m-1}_{j=0}X^{-1/2}_{w}C^{*}_{wj}X^{1/2}_{wj}X^{-1}_{wj}X^{1/2}_{wj}C_{wj}X^{-1/2}_{w}\\
 & =X^{-1/2}_{w}\left(\sum^{m-1}_{j=0}C^{*}_{wj}C_{wj}\right)X^{-1/2}_{w}=X^{-1}_{w}.
\end{aligned}
\]
Thus $X^{-1}$ is a range martingale for $E^{X}$. Applying \prettyref{prop:6-5}
with $Y=X^{-1}$, the combined martingale is the identity family.
Hence 
\[
\left(E^{X}\right)^{X^{-1}}=E.
\]
\end{proof}

\begin{rem}
\label{rem:6-7} In the scalar case, this reduces to the classical
change of measure by a positive martingale. Let $\nu$ be a scalar
probability measure on $\Omega_{m}$, and write 
\[
\nu_{wj}=p_{w}\left(j\right)\nu_{w}.
\]
A strictly positive normalized martingale is a family $x_{w}>0$,
bounded above and below away from zero, satisfying 
\[
x_{w}=\sum^{m-1}_{j=0}p_{w}\left(j\right)x_{wj},\qquad x_{\varnothing}=1.
\]
The transformed measure $\nu^{x}$ has cylinder values 
\[
\nu^{x}_{w}=x_{w}\nu_{w}.
\]
Its local transition probabilities are 
\[
p^{x}_{w}\left(j\right)=p_{w}\left(j\right)\frac{x_{wj}}{x_{w}}.
\]
This is the scalar form of \prettyref{eq:6-3}. Thus \prettyref{thm:6-3}
is the operator-valued version of the Doob change of measure on a
filtered tree. 
\end{rem}

\section{Quadratic variation}\label{sec:7}

We now attach a square-function to the range martingales from \prettyref{sec:4}.
Square functions and quadratic variation are classical tools in martingale
theory, going back to work on martingale transforms and square-function
estimates \cite{MR208647,MR268966}. Later developments include conditioned
square functions and noncommutative martingale inequalities in noncommutative
$L^{p}$-spaces \cite{MR1482934,MR1929141,MR2319715,MR2448025,MR4072235,MR4521734,MR4673616}.
Variance-type quantities for quantum random variables and operator-valued
measures have been studied from several related directions, including
quantum conditional expectation, randomness and noise in quantum measurements,
operator-valued variance, and integration with respect to positive
operator-valued measures \cite{MR2953266,MR2907638,MR3415711,MR4137283}.
In the present section, the variance terms are attached locally to
the range space martingale relation. The positivity will be proved
directly from the range space isometries, rather than by appealing
to a general Schwarz inequality.

Let $X=\left\{ X_{w}\right\} $ be a bounded self-adjoint range martingale
for $E$. Thus 
\[
X_{w}=\sum^{m-1}_{j=0}C^{*}_{wj}X_{wj}C_{wj}
\]
for every finite word $w$.

Recall from \prettyref{sec:3} that the edge contractions at $w$
form the isometry 
\[
J_{w}:H_{w}\to\bigoplus^{m-1}_{j=0}H_{wj},\qquad J_{w}x=\left(C_{w0}x,\ldots,C_{w,m-1}x\right).
\]
If 
\[
Y_{w}=\bigoplus^{m-1}_{j=0}X_{wj}
\]
on $\bigoplus^{m-1}_{j=0}H_{wj}$, then the martingale relation can
be written in the compressed form 
\[
X_{w}=J^{*}_{w}Y_{w}J_{w}.
\]

\begin{defn}
\label{def:7-1} The local variance of $X$ at $w$ is 
\begin{equation}
\Gamma_{w}\left(X\right)=\sum^{m-1}_{j=0}C^{*}_{wj}X^{2}_{wj}C_{wj}-X^{2}_{w}.\label{eq:7-1}
\end{equation}
\end{defn}

\begin{lem}
\label{lem:7-2} For every bounded self-adjoint range martingale $X$
and every finite word $w$, 
\[
\Gamma_{w}\left(X\right)\geq0.
\]
More explicitly, 
\begin{equation}
\Gamma_{w}\left(X\right)=J^{*}_{w}Y_{w}\left(I-J_{w}J^{*}_{w}\right)Y_{w}J_{w}.\label{eq:7-2}
\end{equation}
\end{lem}

\begin{proof}
Since 
\[
J_{w}x=\left(C_{w0}x,\ldots,C_{w,m-1}x\right),
\]
we have 
\[
J^{*}_{w}Y^{2}_{w}J_{w}=\sum^{m-1}_{j=0}C^{*}_{wj}X^{2}_{wj}C_{wj}.
\]
Also, by the martingale identity, 
\[
X^{2}_{w}=\left(J^{*}_{w}Y_{w}J_{w}\right)^{2}=J^{*}_{w}Y_{w}J_{w}J^{*}_{w}Y_{w}J_{w}.
\]
Subtracting gives 
\[
\Gamma_{w}\left(X\right)=J^{*}_{w}Y_{w}\left(I-J_{w}J^{*}_{w}\right)Y_{w}J_{w}.
\]
Since $J_{w}$ is an isometry, $J_{w}J^{*}_{w}$ is the orthogonal
projection onto $J_{w}H_{w}$. Hence $I-J_{w}J^{*}_{w}$ is positive,
and \prettyref{eq:7-2} gives $\Gamma_{w}\left(X\right)\geq0$. 
\end{proof}

For $n\geq0$, write 
\[
X^{\left(n\right)}=\bigoplus_{\left|w\right|=n}X_{w}\in B\left(K_{n}\right).
\]
Define the level-$n$ square operator on $H$ by 
\begin{equation}
Q_{n}\left(X\right)=J^{*}_{0,n}\left(X^{\left(n\right)}\right)^{2}J_{0,n}.\label{eq:7-3}
\end{equation}
Equivalently, by \prettyref{lem:3-2}, 
\begin{equation}
Q_{n}\left(X\right)=\sum_{\left|w\right|=n}E^{1/2}_{w}X^{2}_{w}E^{1/2}_{w}.\label{eq:7-4}
\end{equation}
Here $X_{w}$ is extended by zero on $H^{\perp}_{w}$ when written
inside $B\left(H\right)$.
\begin{thm}
\label{thm:7-3} Let $X=\left\{ X_{w}\right\} $ be a bounded self-adjoint
range martingale. Then 
\begin{equation}
Q_{n+1}\left(X\right)-Q_{n}\left(X\right)=\sum_{\left|w\right|=n}E^{1/2}_{w}\Gamma_{w}\left(X\right)E^{1/2}_{w}\label{eq:7-5}
\end{equation}
for every $n\geq0$. In particular, $\left\{ Q_{n}\left(X\right)\right\} _{n\geq0}$
is an increasing sequence of positive operators. If 
\[
M=\sup_{w}\left\Vert X_{w}\right\Vert ,
\]
then 
\[
0\leq Q_{n}\left(X\right)\leq M^{2}I
\]
for every $n$. Hence $Q_{n}\left(X\right)$ has a strong limit. 
\end{thm}

\begin{proof}
Let 
\[
Y_{n}=X^{\left(n\right)}
\]
on $K_{n}$. Since $X$ is a range martingale, 
\[
Y_{n}=J^{*}_{n}Y_{n+1}J_{n}.
\]
Therefore 
\[
Q_{n}\left(X\right)=J^{*}_{0,n}Y^{2}_{n}J_{0,n}=J^{*}_{0,n}J^{*}_{n}Y_{n+1}J_{n}J^{*}_{n}Y_{n+1}J_{n}J_{0,n}.
\]
On the other hand, 
\[
Q_{n+1}\left(X\right)=J^{*}_{0,n}J^{*}_{n}Y^{2}_{n+1}J_{n}J_{0,n}.
\]
Subtracting gives 
\[
Q_{n+1}\left(X\right)-Q_{n}\left(X\right)=J^{*}_{0,n}J^{*}_{n}Y_{n+1}\left(I-J_{n}J^{*}_{n}\right)Y_{n+1}J_{n}J_{0,n}.
\]
The operator inside this expression is block diagonal over the words
$w$ with $\left|w\right|=n$, and its $w$-block is 
\[
J^{*}_{w}Y_{w}\left(I-J_{w}J^{*}_{w}\right)Y_{w}J_{w}=\Gamma_{w}\left(X\right)
\]
by \prettyref{lem:7-2}. Using \prettyref{lem:3-2}, this gives \prettyref{eq:7-5}.

The right side of \prettyref{eq:7-5} is positive by \prettyref{lem:7-2},
so $Q_{n}\left(X\right)$ is increasing. Also, 
\[
0\leq X^{2}_{w}\leq M^{2}I_{H_{w}}
\]
for every $w$. Hence 
\[
0\leq Q_{n}\left(X\right)\leq M^{2}\sum_{\left|w\right|=n}E_{w}=M^{2}I.
\]
The existence of the strong limit follows from the monotone convergence
theorem for bounded increasing sequences of self-adjoint operators. 
\end{proof}

\begin{defn}
\label{def:7-4} The quadratic variation of $X$ at level $n$ is
\[
\left[X\right]_{n}=Q_{n}\left(X\right)-X^{2}_{\varnothing}.
\]
The terminal quadratic variation is the strong limit 
\[
\left[X\right]_{\infty}=\lim_{n\to\infty}\left[X\right]_{n}.
\]
\end{defn}

Since $Q_{0}\left(X\right)=X^{2}_{\varnothing}$, \prettyref{thm:7-3}
gives 
\[
\left[X\right]_{n}=\sum^{n-1}_{r=0}\sum_{\left|w\right|=r}E^{1/2}_{w}\Gamma_{w}\left(X\right)E^{1/2}_{w}.
\]
Thus $\left[X\right]_{\infty}$ is the accumulated local variance
of the martingale along the tree.

We now relate this square-function to the local sharpness of the splitting
operators. Fix a finite word $w$, and let 
\[
Q_{wj}
\]
be the orthogonal projection of 
\[
\bigoplus^{m-1}_{r=0}H_{wr}
\]
onto the $j$-th summand $H_{wj}$. Then 
\[
J^{*}_{w}Q_{wj}J_{w}=C^{*}_{wj}C_{wj}=A^{\left(j\right)}_{w}.
\]
The same calculation as in \prettyref{lem:7-2}, applied to the child-level
projection $Q_{wj}$, gives 
\[
J^{*}_{w}Q_{wj}\left(I-J_{w}J^{*}_{w}\right)Q_{wj}J_{w}=A^{\left(j\right)}_{w}-(A^{\left(j\right)}_{w})^{2}.
\]
Thus the local variance of the child coordinate projection is 
\[
A^{\left(j\right)}_{w}-(A^{\left(j\right)}_{w})^{2}.
\]
Summing over the children gives the local sharpness loss 
\begin{equation}
N_{w}=I_{H_{w}}-\sum^{m-1}_{j=0}(A^{\left(j\right)}_{w})^{2}=\sum^{m-1}_{j=0}\left(A^{\left(j\right)}_{w}-(A^{\left(j\right)}_{w})^{2}\right).\label{eq:7-6}
\end{equation}
This connects the square-function calculation with the familiar distinction
between general POVMs and sharp, projection-valued observables \cite{MR2817979,MR2770378,MR4112227}.
\begin{prop}
\label{prop:7-5} For every finite word $w$, 
\[
N_{w}\geq0.
\]
Moreover, 
\[
N_{w}=0
\]
if and only if $A^{\left(j\right)}_{w}$ is an orthogonal projection
for every $0\leq j\leq m-1$. Consequently, $E$ is projection-valued
if and only if $N_{w}=0$ for every finite word $w$. 
\end{prop}

\begin{proof}
Since $0\leq A^{\left(j\right)}_{w}\leq I_{H_{w}}$, 
\[
A^{\left(j\right)}_{w}-(A^{\left(j\right)}_{w})^{2}\geq0.
\]
Hence $N_{w}\geq0$.

If $N_{w}=0$, then each positive summand in \prettyref{eq:7-6} is
zero. Therefore 
\[
A^{\left(j\right)}_{w}=(A^{\left(j\right)}_{w})^{2}
\]
for every $j$, so each $A^{\left(j\right)}_{w}$ is an orthogonal
projection. The converse is immediate.

Assume now that $N_{w}=0$ for every finite word $w$. Then each local
splitting operator $A^{\left(j\right)}_{w}$ is a projection. Since
\[
\sum^{m-1}_{j=0}A^{\left(j\right)}_{w}=I_{H_{w}},
\]
these projections are pairwise orthogonal. Indeed, for $i\neq j$,
\[
A^{\left(i\right)}_{w}\leq I_{H_{w}}-A^{\left(j\right)}_{w},
\]
and hence 
\[
A^{\left(i\right)}_{w}A^{\left(j\right)}_{w}=0.
\]

We prove by induction on $\left|w\right|$ that every cylinder value
$E_{w}$ is a projection. This is true for $w=\varnothing$, since
$E_{\varnothing}=I$. Suppose $E_{w}$ is a projection. Then 
\[
H_{w}=E_{w}H
\]
and $E^{1/2}_{w}=E_{w}$. Hence 
\[
E_{wj}=E_{w}A^{\left(j\right)}_{w}E_{w}=A^{\left(j\right)}_{w}
\]
as an operator on $H$, where $A^{\left(j\right)}_{w}$ is extended
by zero on $H^{\perp}_{w}$. Thus every $E_{wj}$ is a projection.
This proves the induction step.

It remains to pass from cylinders to Borel sets. For finite words
$u$ and $v$, the cylinder projections satisfy 
\[
E_{u}E_{v}=E\left(\left[u\right]\cap\left[v\right]\right).
\]
Indeed, if one word extends the other, this follows from the order
relation between the corresponding cylinder projections. If the two
words are incompatible, the corresponding cylinder sets are disjoint,
and the projections are orthogonal.

By finite additivity, the same identity holds for finite unions of
cylinders. The monotone class argument used in the proof of \prettyref{prop:2-4}
then gives 
\[
E\left(B\cap C\right)=E\left(B\right)E\left(C\right)
\]
for all Borel sets $B,C\subseteq\Omega_{m}$. Taking $C=B$, we get
\[
E\left(B\right)^{2}=E\left(B\right)
\]
for every Borel set $B$. Since $E\left(B\right)$ is positive, it
is an orthogonal projection. Thus $E$ is projection-valued.

The converse is immediate from the first part of the proof. If $E$
is projection-valued, then every local splitting operator is a projection,
and hence $N_{w}=0$ for every $w$. 
\end{proof}

The local sharpness terms are thus instances of the same variance
calculus as \prettyref{thm:7-3}. A PVM is the case where the coordinate
projections have zero local variance at every vertex. General POVMs
allow positive local variance, and $\left[X\right]_{\infty}$ measures
how such variance accumulates for an arbitrary bounded self-adjoint
range martingale.

\bibliographystyle{amsalpha}
\bibliography{ref}

\end{document}